\documentclass[11pt,a4paper]{amsart} 

\usepackage{latexsym}
\usepackage{amsmath}
\usepackage{amsthm}
\usepackage{latexsym}
\usepackage{amssymb}
\usepackage[a4paper , lmargin = {3cm} , rmargin = {3cm} , tmargin = {2.5cm} , bmargin = {2.5cm} ]{geometry}

\usepackage{hyperref}

\theoremstyle{plain}
\newtheorem*{mainthm}{Main Theorem}
\newtheorem*{thmA}{Theorem A}
\newtheorem*{thmB}{Theorem B}
\newtheorem*{thmC}{Theorem C}
\newtheorem*{thmD}{Theorem D}
\newtheorem*{thmF}{Theorem F}
\newtheorem{lemma}{Lemma}
\newtheorem*{propE}{Proposition E}
\newtheorem{prop}{Proposition}
\newtheorem*{cor}{Corollary}
\newtheorem*{claim}{Claim}

\newtheorem{remark}{Remark}

\title[]{Conjugating automorphisms of graph products: Kazhdan's property (T) and SQ-universality}
\author{Anthony Genevois and Olga Varghese}
\thanks{A. G. was supported by a public grant as part of the Fondation Math\'ematique Jacques Hadamard. \\O. V. was supported by the Deutsche 
Forschungsgemeinschaft (DFG, German Research Foundation) under Germany's 
Excellence Strategy -EXC 2044-, Mathematics M\"unster: Dynamics-Geometry-Structure}
\date{\today}

\address{Anthony Genevois\\
D\'epartement de Math\'ematiques B\^atiment 307\\
Facult\'e des Sciences d'Orsay Universit\'e Paris-Sud\\
F-91405 Orsay (France)}
\email{anthony.genevois@math.u-psud.fr}

\address{Olga Varghese\\
Department of Mathematics\\
M\"unster University\\ 
Einsteinstra\ss e 62\\
48149 M\"unster (Germany)}
\email{olga.varghese@uni-muenster.de}

\begin{document}
\pagenumbering{arabic}
\begin{abstract}
An automorphism of a graph product of groups is \emph{conjugating} if it sends each factor to a conjugate of a factor (possibly different). In this article, we determine precisely when the group of conjugating automorphisms of a graph product satisfies Kazhdan's property (T) and when it satisfies some vastness properties including SQ-universality. 
\end{abstract}
\maketitle

Given a simplicial graph $\Gamma$ and a collection of groups $\mathcal{G} = \{ G_u \mid u \in V(\Gamma) \}$ indexed by the vertex-set $V(\Gamma)$ of $\Gamma$, the \emph{graph product} $\Gamma \mathcal{G}$ is defined as the quotient
$$ \left( \underset{u \in V(\Gamma)}{\ast} G_u \right) / \langle \langle [g,h]=1, \ g \in G_u, h \in G_v, \{u,v\} \in E(\Gamma) \rangle \rangle$$
where $E(\Gamma)$ denotes the edge-set of $\Gamma$. Two cases to keep in mind is that the graph product $\Gamma \mathcal{G}$ reduces to the free product of $\mathcal{G}$ if $\Gamma$ does not contain any edge, and it reduces to the direct sum of $\mathcal{G}$ if any two vertices of $\Gamma$ are linked by an edge. Usually, one says that graph products interpolate between free products and direct sums. In particular, they provide a convenient framework in which studying right-angled Artin groups (when factors are infinite cyclic) and right-angled Coxeter groups (when factors are cyclic of order two) in a unified way. 

In this article, we are interested in \emph{conjugating automorphisms} of graph products, ie., automorphisms sending factors to conjugates of factors. Groups of conjugating automorphisms of (particular) graph products turn out to appear in different places in the literature. For free groups, these automorphism groups have a very nice topological interpretation, as they coincide with motion groups of circles in the Euclidean spaces \cite{Motion} (see also \cite{Loop} and references therein). Conjugating automorphisms also appear naturally when studying free products \cite{FRI,FRII}, and have been studied independently \cite{SymmetricFree}. In addition to their own interest, studying conjugating automorphisms have been a key step in the understanding of the global automorphism group in several places; see for instance \cite{Tits} for right-angled Coxeter groups and \cite{Laurence} for right-angled Artin groups. Also, the group of conjugating automorphisms turns out to coincides with (a finite-index subgroup of) the global automorphism group in some cases. For instance, as proved in \cite{Genevois}, the group of conjugating automorphisms of a graph product of finite groups has finite index in the entire automorphism group. 

Given a graph product $\Gamma \mathcal{G}$, we denote by $\mathrm{ConjAut}(\Gamma \mathcal{G})$ the group of its conjugating automorphisms, and by  $\mathrm{PConjAut}(\Gamma \mathcal{G})$ its subgroup whose elements send each factor to a conjugate of itself. Also, we emphasize that, in all the article, we use the following convention: factors of graph products (also referred to as \emph{vertex-groups}) are always supposed to be non-trivial. 

The main result of our article is the following:

\begin{mainthm}
Let $\Gamma$ be a simplicial graph and $\mathcal{G}$ a collection of groups indexed by $V(\Gamma)$. If $\Gamma$ is not complete, then $\mathrm{PConjAut}(\Gamma \mathcal{G})$ admits a quotient $Q$ which acts fixed-point freely on a simplicial tree $T$. Moreover, if $\Gamma$ is finite and if the vertex-groups are all finitely generated, then the action $Q \curvearrowright T$ admits a WPD isometry. 
\end{mainthm}

Recall from \cite{MinasyanOsin} that, if a group $G$ acts on a simplicial tree $T$, then $g \in G$ is \emph{WPD} (for \emph{Weakly Properly Discontinuous}) if it is loxodromic and if its axis contains two vertices $u$ and $v$ (not necessarily distinct) such that $\mathrm{stab}_G(u) \cap \mathrm{stab}_G(v)$ is finite.

We deduce many interesting consequences from our theorem. First, we deduce some information regarding Kazhdan's Property (T). Examples of groups with Kazhdan's property (T) are the general linear groups ${\rm GL}_n(\mathbb{Z})$ for $n\geq 3$ (\cite[4.2.5]{Bekka}) and the automorphism groups of free groups ${\rm Aut}(F_n)$ for $n\geq 5$ (\cite{Kielak}, \cite{Kaluba}). 
More generally, an interesting question is the following:
For which shape of the graph $\Gamma$ does the automorphism group ${\rm Aut}(\Gamma \mathcal{G})$ have Kazhdan's property (T)? There are some results in the literature regarding (outer) automorphism groups of graph products of finite abelian groups in \cite[Thm. 1.4]{Leder}, \cite[Cor. 3]{Sale}, \cite[Thm. C]{Varghese}, of right-angled Artin groups in \cite[Cor. 1.8]{Aramayona} and of graph products of arbitrary vertex-groups in \cite[Thm. F]{Genevois}. 

Our main result in this direction is:

\begin{thmA}
Let $\Gamma$ be a finite simplicial graph and $\mathcal{G}$ a collection of groups indexed by $V(\Gamma)$. Then $\mathrm{ConjAut}(\Gamma \mathcal{G})$ satisfies Kazhdan's Property (T) if and only if $\Gamma$ is complete and $G_u/Z(G_u)$ satisfies Kazhdan's Property (T) for every $u \in V(\Gamma)$.
\end{thmA}

Because Property (T) is stable under commensurability \cite[Theorem 1.7.1]{Bekka} and because the group of conjugating automorphisms of a graph product of finite groups has finite-index in the entire automorphism group \cite{Genevois}, we get that: 

\begin{cor}
Let $\Gamma$ be a finite simplicial graph and $\mathcal{G}$ a collection of finite groups indexed by $V(\Gamma)$. Then $\mathrm{Aut}(\Gamma \mathcal{G})$ satisfies Kazhdan's Property (T) if and only if $\Gamma \mathcal{G}$ is finite.
\end{cor}
A good reference for the above corollary is \cite{LV}. We also deduce from the combination of Theorem A with \cite[Theorem 3.23]{Genevois} the following statement, which generalises \cite[Theorem~F]{Genevois}:

\begin{cor}
Let $\Gamma$ be a simplicial graph and $\mathcal{G}$ a collection of groups indexed by $V(\Gamma)$. If $\Gamma$ is a molecular graph, then $\mathrm{Aut}(\Gamma \mathcal{G})$ satisfies Kazhdan's Property (T) if and only if $\Gamma$ is a single vertex or a single edge and if $G_u/Z(G_u)$ satisfies Kazhdan's Property (T) for every $u \in V(\Gamma)$.
\end{cor}

Recall from \cite{Genevois} that a graph is \emph{molecular} if it is finite, connected, leafless and if it has girth $\geq 5$ (ie., it is triangle- and square-free). 

Our second main application deals with SQ-universality. Recall that a group $G$ is \emph{SQ-universal} if, for every countable group $Q$, there exists $N \lhd G$ such that $Q$ embeds into $G/N$. Interestingly, a SQ-universal group has to contain a non-abelian free subgroup and uncountably many normal subgroups. The last observation motivates the idea that being SQ-universal can be thought of as a strong negation of being simple. The first known examples of SQ-universal groups are free groups, so that \emph{large groups} (ie., groups containing finite-index subgroups which surject onto a non-abelian free subgroup) are also SQ-universal. Now, thanks to the theory of acylindrically hyperbolic groups, we know that many groups turn out to be SQ-universal; see \cite{OsinAcyl} and references therein for more information. 

\begin{thmB}
Let $\Gamma$ be a finite simplicial graph and $\mathcal{G}$ a collection of finitely generated groups indexed by $V(\Gamma)$. Decompose $\Gamma$ as a join $\Lambda \ast \Xi$ where $\Lambda$ is the subgraph generated by the vertices which are adjacent to all the vertices of $\Gamma$. Then $\mathrm{ConjAut}(\Gamma \mathcal{G})$ is SQ-universal if and only if $\Xi$ contains a vertex which is not labelled by $\mathbb{Z}_2$, or $\Lambda$ contains a vertex $u$ such that $G_u/Z(G_u)$ is SQ-universal, or $\Xi$ does not decompose as the join of a complete graph with copies of the disjoint union of two single vertices.
\end{thmB}

Actually, the strategy used to prove Theorem B can be also applied to other \emph{vastness properties}, including virtually having an infinite-dimensional space of quasimorphisms and not being boundedly generated.

Recall that, given a group $G$, a \emph{quasimorphism} is a map $\varphi : G \to \mathbb{R}$ such that there exists some constant $C \geq 0$ so that
$$| \varphi(gh) - \varphi(g)- \varphi(h) | \leq C \ \text{for every $g,h \in G$}.$$
Such a quasimorphism is \emph{homogeneous} if $\varphi(g^k)=k \cdot \varphi(g)$ for every $g \in G$ and $k \in \mathbb{Z}$. Notice that the set of all the homogeneous quasimorphisms of our group $G$ is naturally a $\mathbb{R}$-vector space, which we denote by $\mathrm{QH}(G)$. Following \cite{Vastness}, we say that $G$ \emph{virtually has many quasimorphisms} if it contains a finite-index subgroup $H$ such that $\mathrm{QH}(H)$ is infinite-dimensional. 

\begin{thmC}
Let $\Gamma$ be a finite simplicial graph and $\mathcal{G}$ a collection of finitely generated groups indexed by $V(\Gamma)$. Decompose $\Gamma$ as a join $\Lambda \ast \Xi$ where $\Lambda$ is the subgraph generated by the vertices which are adjacent to all the vertices of $\Gamma$. Then $\mathrm{ConjAut}(\Gamma \mathcal{G})$ virtually has many quasimorphisms if and only if $\Xi$ contains a vertex which is not labelled by $\mathbb{Z}_2$, or $\Lambda$ contains a vertex $u$ such that $G_u/Z(G_u)$ virtually has many quasimorphisms, or $\Xi$ does not decompose as the join of a complete graph with copies of the disjoint union of two single vertices.
\end{thmC}

Next, recall that a group $G$ is \emph{boundedly generated} if there exist finitely many elements $g_1, \ldots, g_n \in G$ such that, for every $g \in G$, there exist integers $m_1, \ldots, m_n \in \mathbb{Z}$ so that $g = g_1^{m_1} \cdots g_n^{m_n}$. 

\begin{thmD}
Let $\Gamma$ be a finite simplicial graph and $\mathcal{G}$ a collection of finitely generated groups indexed by $V(\Gamma)$. Decompose $\Gamma$ as a join $\Lambda \ast \Xi$ where $\Lambda$ is the subgraph generated by the vertices which are adjacent to all the vertices of $\Gamma$. Then $\mathrm{ConjAut}(\Gamma \mathcal{G})$ is boundedly generated if and only if the vertices of $\Xi$ are all labelled by $\mathbb{Z}_2$, for every vertex $u$ of $\Lambda$ the group $G_u/Z(G_u)$ is boundedly generated, and $\Xi$ decomposes as the join of a complete graph with copies of the disjoint union of two single vertices.
\end{thmD}

Because the group of conjugating automorphisms of a graph product of finite groups turns out to have finite-index in the entire automorphism group \cite{Genevois}, we deduce the following statement from the theorems above:

\begin{propE}
Let $\Gamma$ be a finite simplicial graph and $\mathcal{G}$ a collection of finite groups indexed by $V(\Gamma)$. Decompose $\Gamma$ as a join $\Lambda \ast \Xi$ where $\Lambda$ is the subgraph generated by the vertices which are adjacent to all the vertices of $\Gamma$. For  $\mathrm{Aut}(\Gamma \mathcal{G})$ the following statements are equivalent:
\begin{enumerate} 
\item[(i)] The subgraph $\Xi$ contains a vertex which is not labelled by $\mathbb{Z}_2$ or $\Xi$ does not decompose as the join of a complete graph with copies of the disjoint union of two single vertices.
\item[(ii)] The automorphism group $\mathrm{Aut}(\Gamma \mathcal{G})$ involves all finite groups. 
\item[(iii)] The automorphism group $\mathrm{Aut}(\Gamma \mathcal{G})$ is SQ-universal.
\item[(iv)] The automorphism group $\mathrm{Aut}(\Gamma \mathcal{G})$ virtually has many quasimorphisms.
\item[(v)] The automorphism group $\mathrm{Aut}(\Gamma \mathcal{G})$ is not boundedly generated.
\item[(vi)] The group $\Gamma \mathcal{G}$ is not virtually free abelian.
\end{enumerate} 
\end{propE}

In this statement, following \cite{Vastness}, we say that a group $G$ \emph{involves all finite groups} if any finite group appears as a quotient of a finite-index subgroup of $G$.

\section*{Proof of the main theorem}

\noindent
Let $\Gamma$ be a simplicial graph and $\mathcal{G}$ a collection of groups indexed by $V(\Gamma)$. Assume that $\Gamma$ is not complete.

\medskip \noindent
Fix two vertices $u,v \in V(\Gamma)$ which are not adjacent. Notice that, as $\mathrm{PConjAut}(\Gamma \mathcal{G})$ sends each vertex-group to a conjugate of itself, the normal closure $Q$ in $\Gamma \mathcal{G}$ of the subgroup $\langle G_w, \ w \in V(\Gamma) \backslash \{u,v\} \rangle$ is preserved by $\mathrm{PConjAut}(\Gamma \mathcal{G})$. Moreover, the inclusion $\langle G_u,G_v \rangle \hookrightarrow \Gamma \mathcal{G}$ induces an isomorphism $\langle G_u,G_v \rangle \to \Gamma \mathcal{G}/Q$, and the subgroup $\langle G_u, G_v \rangle$ is naturally isomorphic to the free product $G_u \ast G_v$. As a consequence, the quotient $\Gamma \mathcal{G} \to \Gamma \mathcal{G}/Q$ induces a homomorphism $\Phi : \mathrm{PConjAut}(\Gamma \mathcal{G}) \to \mathrm{PConjAut}(G_u \ast G_v)$. Notice that $\Phi( \mathrm{Inn}(\Gamma \mathcal{G})) = \mathrm{Inn}(G_u \ast G_v)$, so that $\mathrm{Inn}(G_u \ast G_v)$ is included into the image of $\Phi$.

\medskip \noindent
The fact that $\mathrm{PConjAut}(\Gamma \mathcal{G})$ acts fixed-point freely on a simplicial tree now follows from the following observation:

\begin{claim}
Let $A,B$ be two non-trivial groups. If $T$ denotes the Bass-Serre tree of the free product $A \ast B$, then
$$\psi \mapsto \left( v \mapsto \text{vertex whose stabiliser is $\psi(\mathrm{stab}(v))$} \right)$$
defines an isometric action $\mathrm{PConjAut}(A \ast B) \curvearrowright T$ which extends $A \ast B \curvearrowright T$ when $A \ast B$ is canonically identified to $\mathrm{Inn}(A \ast B)$. 
\end{claim}

\noindent
We refer to  \cite[\S 4, Thm. 7]{Serre} for more information on Bass-Serre trees. A proof of the previous claim can be found in  \cite{Karrass} and in \cite[Thm. 1]{Pettet} in greater generality. We include here the sketch of a direct proof for reader's convenience.

\begin{proof}[Sketch of proof of the claim.]
The vertices of the Bass-Serre tree $T$ are the cosets of $A$ and $B$. Because $A$ and $B$ are malnormal subgroups of $A \ast B$ (ie., for every $g \in A \ast B$, the intersection $A \cap gBg^{-1}$ is always trivial and the intersection $A \cap gAg^{-1}$ is non-trivial only when $g \in A$), there is a natural bijection between the cosets of $A$ and $B$ and their conjugates. By definition of $\mathrm{PConjAut}(\Gamma \mathcal{G})$, it follows that 
$$\psi \mapsto \left( v \mapsto \text{vertex whose stabiliser is $\psi(\mathrm{stab}(v))$} \right)$$
defines an action of $\mathrm{PConjAut}(\Gamma \mathcal{G}))$ on the vertices of $T$. It remains to show that the action preserves the adjacency relation of the vertices of $T$.

\medskip \noindent
Notice that, if $u,v \in T$ are two vertices, then the geodesic $[u,v]$ is a fundamental domain for the action of the subgroup $\langle \mathrm{stab}(u) , \mathrm{stab}(v) \rangle \leq A \ast B$ on $T$. Therefore, since $A \ast B$ acts on $T$ with a single orbit of edges, the two vertices $u$ and $v$ turn out to be adjacent if and only if  $\langle \mathrm{stab}(u) , \mathrm{stab}(v) \rangle = A \ast B$. Because this characterisation is purely algebraic, it must be preserved by automorphisms of $A \ast B$, concluding the proof. 
\end{proof}

\noindent
Thus, we have proved the first assertion of our theorem. The second assertion now follows from the following observation:

\begin{claim}
Let $A,B$ be two finitely generated groups. Then the action of $\mathrm{PConjAut}(A \ast B)$ on the Bass-Serre tree $T$ of $A \ast B$ contains a WPD isometry.
\end{claim}

\noindent
Fix two generating sets $\{s_1, \ldots, s_n\}$ and $\{r_1, \ldots, r_n\}$ of $A$ and $B$ respectively, such that $r_i$ and $s_i$ are non-trivial for every $1 \leq i \leq n$. Set
$$g := s_1r_1 \cdot s_2r_2 \cdots s_nr_n \in A \ast B.$$
We claim that the inner automorphism $\iota(g) \in \mathrm{PConjAut}(A \ast B)$ defines a WPD isometry of $T$. First of all, notice that, if $\sigma \subset T$ denotes the segment
$$A, \ s_1\cdot B, \ s_1r_1 \cdot A, \ s_1r_1s_2 \cdot B, \ s_1r_1s_2r_2 \cdot A, \ldots,  \ s_1r_1 \cdots s_{n-1}r_{n-1}s_{n}r_n \cdot A=g \cdot A$$ 
then $\gamma:= \bigcup\limits_{k \in \mathbb{Z}} g^k \sigma$ defines a geodesic on which $g$ acts by translations. Notice also that the penultimate vertex of $\sigma$ is 
$$s_1r_1 \cdots s_{n-1}r_{n-1}s_{n} \cdot B = s_1r_1 \cdots s_{n-1}r_{n-1}s_{n}r_n \cdot B= g \cdot B,$$
so that $B$ belongs to $\gamma$. Therefore, $\gamma$ is an axis of $\iota(h)$ which contains the vertices $A$, $B$, $g \cdot A$ and $g \cdot B$. As a consequence of \cite[Corollary 4.3]{MinasyanOsin}, in order to show that $\iota(h)$ is WPD, it is sufficient to verify that the intersection
$$\mathrm{stab}(A) \cap \mathrm{stab}(B) \cap \mathrm{stab}(g \cdot A) \cap \mathrm{stab}(g \cdot B)$$
is finite, where the stabilisers are taken with respect to the action of $\mathrm{PConjAut}(A \ast B)$. Otherwise saying, we want to show that
$$F:= \{ \varphi \in \mathrm{PConjAut}(A \ast B) \mid \varphi(A)=A, \ \varphi(B)=B, \ \varphi(g \cdot A)= g \cdot A, \ \varphi(g \cdot B)=g \cdot B \}$$
is finite. So let $\varphi \in F$. Because $g \cdot A = \varphi(g \cdot A)=\varphi(g) \cdot A$ and $g \cdot B = \varphi(g \cdot B) = \varphi(g) \cdot B$, we deduce that $g^{-1} \varphi(g) \in A \cap B= \{1\}$, hence $\varphi(g)=g$. Because $\varphi$ stabilises both $A$ and $B$, it follows that
$$s_1r_1 \cdot s_2r_2 \cdots s_nr_n = g = \varphi(g)= \varphi(s_1)\varphi(r_1) \cdot \varphi(s_2) \varphi(r_2) \cdots \varphi(s_n) \varphi(r_n)$$
is an equality between two alternating words, so we must have $\varphi(s_i)=s_i$ and $\varphi(r_i)=r_i$ for every $1 \leq i \leq n$. But $\{s_1, \ldots, s_n\}$ and $\{r_1, \ldots, r_n\}$ generate $A$ and $B$ respectively, so $\varphi$ must be the identity. Thus, we have proved that $F=\{1\}$, which concludes the proof of our claim. 

\begin{remark}
It is worth noticing that the assumption on the finite generation of vertex-groups can be weakened. Indeed, we only need each vertex-group to contain a finite set such that any automorphism fixing pointwise this set has to be the identity. Of course, taking a finite generating set works, but groups which are not finitely generated may also contain such sets. For instance, any automorphism of $\mathbb{Q}$ fixing $\{1\}$ must be the identity.
\end{remark}

\section*{Proofs of Theorems A, B, C and D}

\noindent
We begin by proving Theorem A about Kazhdan's Property (T).

\begin{proof}[Proof of Theorem A]
If $\Gamma$ is not complete, then it follows from our main theorem that $\mathrm{PConjAut}(\Gamma \mathcal{G})$ acts fixed-point freely on a simplicial tree. Because Property (T) is stable under commensurability \cite[Theorem 1.7.1]{Bekka} and that a group acting fixed-point freely on a simplicial tree cannot satisfy Property (T) \cite{Watatani}, we conclude that $\mathrm{ConjAut}(\Gamma \mathcal{G})$ does not satisfy Property~(T).

\medskip \noindent
Conversely, if $\Gamma$ is a complete graph, then 
$$\mathrm{PConjAut}(\Gamma \mathcal{G})= \bigoplus\limits_{u \in V(\Gamma)} \mathrm{Inn}(G_u) \simeq \bigoplus\limits_{u \in V(\Gamma)} G_u/Z(G_u).$$
Because Property (T) is stable under commensurability and because a direct sum satisfies Property (T) if and only if so do its factors \cite[Proposition 1.7.8]{Bekka}, we conclude that $\mathrm{ConjAut}(\Gamma \mathcal{G})$ satisfies Property (T) if and only if so does $G_u/Z(G_u)$ for every $u \in V(\Gamma)$. 
\end{proof}

\noindent
Next, let us turn to the proof of Theorem B and its corollaries. From now on, we fix a group property $\mathcal{P}$ satisfying the following conditions:
\begin{itemize}
	\item[(i)] acylindrically hyperbolic groups satisfy $\mathcal{P}$; 
	\item[(ii)] if a group has a quotient satisfying $\mathcal{P}$, then it satisfies $\mathcal{P}$ as well;
	\item[(iii)] a group satisfies $\mathcal{P}$ if and only if so do its finite-index subgroups;
	\item[(iv)] the direct sum of two groups satisfies $\mathcal{P}$ if and only if so does one of the two factors;
	\item[(v)] free abelian groups do not satisfy $\mathcal{P}$.
\end{itemize}
We refer to \cite{OsinAcyl} for more information on acylindrically hyperbolic groups. The only thing we need to know that, if a group which is not virtually cyclic acts on a simplicial tree with a WPD isometry, then it has to be acylindrically hyperbolic.

\medskip \noindent
Our goal now is to show the following statement:

\begin{thmF}
Let $\Gamma$ be a finite simplicial graph and $\mathcal{G}$ a collection of finitely generated groups indexed by $V(\Gamma)$. Decompose $\Gamma$ as a join $\Lambda \ast \Xi$ where $\Lambda$ is the subgraph generated by the vertices which are adjacent to all the vertices of $\Gamma$. Then $\mathrm{ConjAut}(\Gamma \mathcal{G})$ satisfies $\mathcal{P}$ if and only if $\Xi$ contains a vertex which is not labelled by $\mathbb{Z}_2$, or $\Lambda$ contains a vertex $u$ such that $G_u/Z(G_u)$ satisfies $\mathcal{P}$, or $\Xi$ does not decompose as the join of a complete graph with copies of the disjoint union of two single vertices.
\end{thmF}

\noindent
We begin by considering the case of right-angled Coxeter groups. Namely:

\begin{prop}\label{prop:RACG}
Let $\Gamma$ be a finite simplicial graph. Then $\mathrm{Aut}(C_\Gamma)$ satisfies $\mathcal{P}$ if and only if $\Gamma$ does not decompose as the join of a complete graph with copies of the disjoint union of two single vertices.
\end{prop}

\noindent
The proposition will be essentially deduced from the following observation, which is a consequence of \cite[Proposition 17.2.1]{Davis}:

\begin{lemma}\label{lem:RACGSQ}
Let $\Gamma$ be a finite simplicial graph. The following assertions are equivalent:
\begin{itemize}
	\item the right-angled Coxeter group $C_\Gamma$ is \emph{large} (ie.,  it contains a finite-index subgroup which surjects onto a non-abelian free subgroup); 
	\item $\Gamma$ does not decompose as the join of a complete graph with copies of the disjoint union of two single vertices;
	\item $C_\Gamma$ decomposes as a direct sum of copies of $\mathbb{Z}_2$ and $\mathbb{D}_\infty$.
\end{itemize}
\end{lemma}

\begin{proof}[Proof of Proposition \ref{prop:RACG}.]
A SIL in $\Gamma$ is the data of two vertices $u,v \in V(\Gamma)$ satisfying $d(u,v) \geq 2$ such that $\Gamma \backslash (\mathrm{link}(u) \cap \mathrm{link}(v) )$ contains at least one connected component which does not contain $u$ nor $v$. Notice that, if $\Gamma$ decomposes as the join of a complete graph with copies of the disjoint union of two single vertices, then $\Gamma$ cannot contain a SIL.

\medskip \noindent
If $\Gamma$ contains a SIL, then $\mathrm{Out}(C_\Gamma)$ is large according to \cite{Sale}, ie., it contains a finite-index subgroup which surjects onto a non-abelian free subgroup. We deduce from the conditions (i), (ii) and (iii) that $\mathrm{Aut}(C_\Gamma)$ has to satisfy $\mathcal{P}$. 

\medskip \noindent
From now on, assume that $\Gamma$ does not contain a SIL. As a consequence of \cite[Theorem~1.4]{GPR}, $\mathrm{Out}(C_\Gamma)$ must be finite, so that $\mathrm{Inn}(C_\Gamma)$ must have finite index in $\mathrm{Aut}(C_\Gamma)$. Notice that, if $\Lambda$ denotes the subgraph of $\Gamma$ generated by the vertices which are adjacent to all the vertices of $\Gamma$, then the center of $C_\Gamma$ is generated by the generators labelling the vertices of $\Lambda$. Therefore, $C_\Gamma$ decomposes as $C_{\Gamma \backslash \Lambda} \oplus \mathbb{Z}_2^{\# V(\Lambda)}$ and $\mathrm{Inn}(C_\Gamma) \simeq C_\Gamma / Z(C_\Gamma) \simeq C_{\Gamma \backslash \Lambda}$. 

\medskip \noindent
We distinguish two cases. If $\Gamma \backslash \Lambda$ does not decompose as the join of disjoint unions of two single vertices, then it follows from Lemma \ref{lem:RACGSQ} and from the conditions (i), (ii) and (iii) that $\mathrm{Aut}(C_\Gamma)$ has to satisfy $\mathcal{P}$. Otherwise, if $\Gamma \backslash \Lambda$ decomposes as the join of disjoint unions of two single vertices, then $C_{\Gamma \backslash \Lambda}$ is a direct sum of infinite dihedral groups. As such a group is virtually abelian, we deduce from the conditions (iii) and (v) that $\mathrm{Aut}(C_\Gamma)$ cannot satisfy~$\mathcal{P}$. 
\end{proof}

\noindent
We are now ready to prove Theorem F.

\begin{proof}[Proof of Theorem F]
If $\Gamma$ contains two non-adjacent vertices $u$ and $v$ which are not both labelled by $\mathbb{Z}_2$, then we have seen in the previous section that $\mathrm{PConjAut}(G_u \ast G_v)$ is a quotient of $\mathrm{PConjAut}(\Gamma \mathcal{G})$ and that it acts on a simplicial tree with a WPD isometry. Because $G_u \ast G_v$ is not virtually cyclic, it follows from \cite{OsinAcyl} that $\mathrm{PConjAut}(G_u \ast G_v)$ must be acylindrically hyperbolic, so that $\mathrm{ConjAut}(\Gamma \mathcal{G})$ has to satisfy $\mathcal{P}$ as a consequence of the conditions (i), (ii) and (iii).

\medskip \noindent
Now, assume that any two non-adjacent vertices of $\Gamma$ are labelled by $\mathbb{Z}_2$. As a consequence, if we decompose $\Gamma$ as a join $\Lambda \ast \Xi$ where $\Lambda$ is the subgraph generated by the vertices with are adjacent to all the vertices of $\Gamma$, then the vertices of $\Xi$ have to be all labelled by $\mathbb{Z}_2$. Therefore, $\mathrm{PConjAut}(\Gamma \mathcal{G})$ is isomorphic to the direct sum 
$$\mathrm{PConjAut}(C_\Xi) \oplus \bigoplus\limits_{u \in V(\Xi)} \mathrm{Inn}(G_u) \simeq  \mathrm{PConjAut}(C_\Xi) \oplus \bigoplus\limits_{u \in V(\Xi)} G_u/Z(G_u).$$
As a consequence of the condition (ii) and (iv), we know that $\mathrm{ConjAut}(\Gamma \mathcal{G})$ satisfies $\mathcal{P}$ if and only if so does one the direct factors above. Therefore, it follows from Proposition \ref{prop:RACG} that $\mathrm{ConjAut}(\Gamma \mathcal{G})$ does not satisfy $\mathcal{P}$ only when $G_u/Z(G_u)$ does not satisfy $\mathcal{P}$ for every $u \in \Xi$ and when $\Xi$ decomposes as the join of a complete graph with copies of the disjoint union of two single vertices.
\end{proof}

\begin{remark}\label{remark}
It is worth noticing that the proof above shows that Theorem F still holds if the condition (i) satisfied by $\mathcal{P}$ is replaced with: if $A$ and $B$ are two groups which are not both cyclic of order two, then $\mathrm{ConjAut}(A \ast B)$ satisfies $\mathcal{P}$. 
\end{remark}

\noindent
Finally, Theorems B, C and D are direct consequences of the combination of Theorem E above with our next statement:

\begin{prop}\label{prop:Vast}
The following group properties satisfy the conditions (i), (ii), (iii), (iv) and (v) above:
\begin{itemize}
	\item being SQ-universal;
	\item virtually having many quasimorphisms;
	\item not being boundedly generated.
\end{itemize}
\end{prop}

\begin{proof}
It is clear that being SQ-universal satisfies the conditions (ii) and (v). The conditions (i), (iii) and (iv) follow from \cite[Theorem 8.1]{OsinAcyl}, \cite{Neumann} and \cite[Lemma 1.12]{Vastness} respectively. 

\medskip \noindent
It is clear that not being boundedly generated satisfies the conditions (ii), (iii), (iv) and (v). About condition (i), we refer to the discussion following \cite[Theorem 2.33]{DGO} and to the references mentioned therein.

\medskip \noindent
Virtually having many quasimorphisms satisfies the conditions (ii), (iii) and (iv), see \cite[Proposition 1.18]{Vastness} for more details. Condition (i) is proved by \cite{BFquasi}. About condition (v), it is sufficient to notice that $\mathbb{Z}$ has a one-dimensional space of quasimorphisms and next to apply condition (iv).
\end{proof}

\noindent
We conclude the article by proving Proposition E.

\begin{proof}[Proof of Proposition E.]
According to \cite[Theorem 3.15]{Genevois}, $\mathrm{ConjAut}(\Gamma \mathcal{G})$ has finite index in $\mathrm{Aut}(\Gamma \mathcal{G})$. Therefore, the equivalences between (i), (iii), (iv) and (v) are immediate consequences of Theorems B, C and D and Proposition \ref{prop:Vast}. 

\medskip \noindent
Next, it follows from \cite[Corollary 1.6]{Vastness} that involving all finite groups satisfies the conditions (ii), (iii) and (iv) above. Condition (v) is also clearly satisfied. Now let $A$ and $B$ be two finite groups which are not both cyclic of order two. Notice that $\mathrm{ConjAut}(A \ast B)$ contains $\mathrm{Inn}(A \ast B) \simeq A \ast B$ as a normal subgroup, which is virtually a finitely generated free group of rank $\geq 2$ according to \cite[Proposition 4, page 6]{Serre}. As a consequence of \cite[Lemma 1.5]{Vastness}, $\mathrm{ConjAut}(A \ast B)$ has to involve all finite groups. Therefore, we deduce from Remark \ref{remark} that the conclusion of Theorem F holds in the realm of graph products of finite groups, proving the equivalence between (i) and (ii) in our proposition.

\medskip \noindent
If we denote by $\langle \Lambda \rangle$ (resp. $\langle \Xi \rangle$) the subgroup of $\Gamma \mathcal{G}$ generated by the vertex-groups labelling the vertices of $\Lambda$ (resp. of $\Xi$), then $\Gamma \mathcal{G}$ decomposes as the direct sum $\langle \Xi \rangle \oplus \langle \Lambda \rangle$. Notice that $\langle \Lambda \rangle$ is finite, so that $\langle \Xi \rangle$ must have finite index in $\Gamma \mathcal{G}$. If $\Xi$ contains a vertex $u$ which is labelled by a group $G_u$ which is not cyclic of order two, then $\Gamma \mathcal{G}$ contains the free product $\langle G_u, G_v \rangle \simeq G_u \ast G_v$ as a subgroup, where $v$ is a vertex of $\Gamma$ which is not adjacent to $u$. But such a group is virtually a non-abelian free group according to \cite[Proposition 4, page 6]{Serre}, so $\Gamma \mathcal{G}$ cannot be virtually abelian. So let us suppose that the vertices of $\Xi$ are all labelled by $\mathbb{Z}_2$. As a consequence, the subgroup $\langle \Xi \rangle$ is isomorphic to the right-angled Coxeter group $C_\Xi$, so that the desired conclusion follows from Lemma \ref{lem:RACGSQ}. 
\end{proof}

\end{document}